\documentclass[a4paper,12pt,reqno]{amsart}
\usepackage{}
\usepackage{amssymb}
\usepackage{mathrsfs,amsmath,amssymb,latexsym,amsfonts, amscd}
\usepackage[all,ps,cmtip]{xy}

\title{On minimal 3-folds of general type with maximal pluricanonical section index}
\author{Meng Chen}

\dedicatory{Dedicated to Prof. Ngaiming Mok on his sixtieth birthday }

\address{\rm School of Mathematical Sciences \& Shanghai Centre for Mathematical Sciences, Fudan University, Shanghai 200433, China}
\email{mchen@fudan.edu.cn}

\newcommand{\bQ}{{\mathbb Q}}
\newcommand{\bP}{{\mathbb P}}
\newcommand{\roundup}[1]{\lceil{#1}\rceil}
\newcommand{\rounddown}[1]{\lfloor{#1}\rfloor}

\newcommand\lrw{\longrightarrow}

\newcommand\OO{{\mathcal{O}}}

\newcommand\caL{{\mathcal{L}}}

\newcommand\bZ{{\mathbb{Z}}}

\newcommand{\lsgeq}{\succcurlyeq}
\newcommand{\lsleq}{\preccurlyeq}

\newtheorem{thm}{Theorem}[section]
\newtheorem{lem}[thm]{Lemma}
\newtheorem{cor}[thm]{Corollary}
\newtheorem{prop}[thm]{Proposition}

\theoremstyle{definition}

\newtheorem{rem}[thm]{Remark}
\theoremstyle{remark}

\begin{document}
\begin{abstract} 
Let $X$ be a minimal 3-fold of general type. The pluricanonical section index $\delta(X)$ is defined to be the minimal integer $m$ so that $P_{m}(X)\geq 2$. According to Chen-Chen, one has either $1\leq \delta(X)\leq 15$ or $\delta(X)=18$.  This note aims to intensively study those with maximal such index. A direct corollary is that the $57$th canonical map of every minimal 3-fold of general type is stably birational. 
\end{abstract}
\maketitle

\pagestyle{myheadings}
\markboth{\hfill M. Chen\hfill}{\hfill Threefolds of general type with maximal ps-index\hfill}
\numberwithin{equation}{section}

\section{\bf Introduction} 

One main task of birational geometry is to study the behavior of pluricanonical maps of projective varieties. In this paper, we restrict our interest to minimal projective 3-folds of general type. Without loss of generality, we may always study a minimal variety of general type over any algebraically closed field $k$ of characteristic $0$. 

Let $X$ be a minimal projective $n$-fold of general type. Traditionally $X$ is always assumed to be $\bQ$-factorial with at worst terminal singularities.  We always denote by $\varphi_{m,X}$ the $m$-th canonical map corresponding to $|mK_X|$, where $K_X$ is a canonical divisor of $X$.  The {\it canonical stability index of $X$} is defined as
$$r_s(X)=\text{min}\{m\in \bZ_{>0}|\ \varphi_{l,X}\ \text{is birational for all}\ l\geq m\}.$$
The {\it $n$-th canonical stability index} $r_n$ is defined as 
$$r_n=\text{sup}\{r_s(X)|\ X\ \text{is a minimal n-fold of general type}\}.$$
Such number $r_n$ has the fundamental importance in explicit birational geometry (see, for instance, Hacon--McKernan \cite[Problem 1.5]{H-M06}). Here are the main results about $r_n$:
\begin{itemize}
\item[$\diamond$] $r_1=3$, as a well-known fact.

\item[$\diamond$] $r_2=5$ according to Bombieri \cite{Bom73}.

\item[$\diamond$] $r_n<+\infty$ for all $n\geq 3$ by Hacon--McKernan \cite{H-M06} and Takayama \cite{Tak06}, independently. 

\item[$\diamond$] $r_3\leq 61$ by Chen--Chen \cite{EXPI,EXPII,EXPIII}. 
\end{itemize}

With regard to the value of $r_3$, Hacon--McKernan \cite[Question 1.6]{H-M06} asked if $r_3=27$ is true.  So far there is no known minimal 3-fold $X$ satisfying $r_s(X)>27$.  In this short note, we are going to improve the known upper bound for $r_3$. 

Let $X$ be a minimal 3-fold of general type. As in Chen--Chen \cite{EXPIII}, the {\it pluricanonical section index }(in short, {\it ps-index}) is defined as follows:
$$\delta(X)=\text{min}\{m\in \bZ_{>0}|\ P_m(X)\geq 2\}.$$
By Chen--Chen \cite[Theorem 1.4]{EXPIII}, we know that either $1\leq \delta(X)\leq 15$ or $\delta(X)=18$.   {}From our experience in studying the pluricanonical  birationality, those 3-folds with largest ps-idex are the main obstacle for us to get better estimation of $r_3$. Thus we concentrate on studying those with maximal ps-index. Here is our main result: 

\begin{thm}\label{57} Let $X$ be a minimal 3-fold of general type with $\delta(X)=18$. Then $r_s(X)\leq 57$.
\end{thm}

Together with Chen--Chen \cite[Theorem 6.2]{EXPIII}, Theorem \ref{57} directly implies the following:
\begin{cor}\label{cc} One has $r_3\leq 57$.
\end{cor}
\medskip

It is interesting to ask what the exact value of  $r_3$ is. However, it seems to be very difficult to get an answer of this since the existence with $11\leq \delta\leq18$ has been in suspense. Theoretically it is clear from Chen-Chen \cite{EXPIII} that any minimal 3-fold $X$ with  $57\leq r_s(X)\leq 61$ must have $\delta(X)=18$.  Thus to study those with maximal ps-index is a very natural choice. It is known that every minimal 3-fold with $\delta=18$ has an induced fibration from $|18K|$ whose general fiber is a surface of general type with $p_g=1$.  A feature of this paper is that three maps $\varphi_{18}$, $\varphi_{24}$ and $\varphi_{36}$ and their mutual actions are considered simultaneously to prove the above mentioned main theorem.  
\bigskip

Throughout we will use the following symbols:
\begin{itemize}
\item[$\diamond$] ``$\sim$'' denotes linear equivalence or ${\mathbb Q}$-linear equivalence;
\item[$\diamond$] ``$\equiv$'' denotes numerical equivalence;
\item[$\diamond$] ``$|M_1|\lsgeq |M_2|$'' (or, equivalently,  ``$|M_2|\lsleq |M_1|$'') means, for linear systems $|M_1|$ and $|M_2|$ on a variety,
$$|M_1|\supseteq|M_2|+\text{(fixed effective divisor)}.$$
\item[$\diamond$] ``$D\leq D'$'' means that $D'-D$ is linearly (or ${\mathbb Q}$-linearly) equivalent to an effective divisor (or effective $\bQ$-divisor)  subject to the context for two divisors (or $\bQ$-divisors) $D$ and $D'$.
\end{itemize}

\section{\bf Preliminaries}

\subsection{Convention}  Let $|D|$ be any linear system of positive dimension on a normal projective variety $Z$. 
\begin{itemize}
\item[(1)]  Denote by $\text{Mov}|D|$ the moving part of $|D|$. 
We say that $|D|$ is {\it composed of a pencil} if $\dim\overline{Im(\Phi_{\text{Mov}|D|})}=1$. 

\item[(2)] A {\it generic irreducible element of $|D|$} means a general member of $\text{Mov}|D|$ when $|D|$ is not composed of a pencil or, otherwise, an irreducible component in a general member of $\text{Mov}|D|$.  

\item[(3)] Let $S_1$, $S_2$ be two different irreducible and reduced divisors which are not contained in the fixed locus of $|D|$. If 
$$\overline{\Phi_{|D|}(S_1)}\not\subseteq \overline{\Phi_{|D|}(S_2)}\ \text{and}\ \overline{\Phi_{|D|}(S_2)}\not\subseteq \overline{\Phi_{|D|}(S_1)},$$ we say that {\it $|D|$ can distinguish $S_1$ and $S_2$}. 
\end{itemize}

\subsection{Induced fibration from $\varphi_{m_0}$} \label{m0}

Let $X$ be a minimal projective 3-fold of general type on which $P_{m_0}(X)\geq 2$ for an integer $m_0>0$.  {}Fix an effective Weil divisor $K_{m_0}\sim m_0K_X$ on $X$. Take successive blow-ups, say $\pi\colon X'\rightarrow X$ along nonsingular centers, such that the
following conditions are satisfied:
\medskip
\begin{itemize}
\item[(i)] $X'$ is nonsingular and projective;

\item[(ii)] the moving part  of $|m_0K_{X'}|$ is base point free and so that $g_{m_0}=\varphi_{m_0,X}\circ\pi$ is a non-constant morphism;

\item[(iii)]  $\pi^*(K_{m_0})\cup\{\pi-\text{exceptional divisors} \}$ has simple normal crossing
supports.
\end{itemize}

Taking the Stein factorization of the morphism $$g_{m_0}\colon X'\longrightarrow \overline{\varphi_{m_0,X}(X)}\subseteq{\mathbb
P}^{P_{m_0}-1},$$
we have  $X'\overset{f_{m_0}}\longrightarrow
\Gamma\overset{s}\longrightarrow \overline{\varphi_{m_0,X}(X)}$ and the following commutative diagram:

\begin{eqnarray*}
\xymatrix{ X' \ar[rr]^{f_{m_0}}\ar[d]_{\pi }\ar[drr]^{g_{m_0}}&& \Gamma \ar[d]^{s}\\
X \ar@{.>}[rr]_{\varphi_{m_0,X}} && \overline{\varphi_{m_0,X}(X)}}
\end{eqnarray*}
We call $f_{m_0}:X'\lrw \Gamma$ an {\it induced fibration from $\varphi_{m_0}$}. 
Denote by $r(X)$ the canonical index of $X$. 
We may write $m_0K_{X'}\sim_{\mathbb Q}\pi^*(m_0K_X)+E_{\pi, m_0}$
where $E_{\pi, m_0}$ is an effective $\pi$-exceptional ${\mathbb
Q}$-divisor. Denote by $|M_{m_0}|$ the moving part of $|m_0K_{X'}|$.  Set $d_{m_0}=
\dim (\Gamma)$. By the Bertini theorem, when $d_{m_0}\geq 2$,  the general member  $F\in |M_{m_0}|$ is irreducible and smooth and we set $a_{m_0}=1$. 
When $d_{m_0}=1$, $M_{m_0}\equiv
a_{m_0}F$, where $a_{m_0}=\deg f_*\OO_{X'}(M_{m_0})$
and $F$ is a general fiber of $f$.  By the above setting, we always have
\begin{equation}m_0\pi^*(K_X)\equiv a_{m_0}F+E_{m_0}'\end{equation}
for some effective $\bQ$-divisor $E_{m_0}'$ on $X'$.

\begin{lem}\label{integ} Keep the above notation.  Denote by $r_X$ the Cartier index of $X$. Then $r_X(\pi^*(K_X)|_F)^2$ is an integer. 
\end{lem}
\begin{proof} Set $\pi_*(F)=F_0$. Then $r_X(\pi^*(K_X)|_F)^2=r_X(K_X^2\cdot F_0)$ which is independent of the choice of the birational modification $\pi$.  Take $\mu:Y\lrw X$ to be a resolution of singularities of $X$ which is assumed to have at worst terminal (hence isolated) singularities. Write $K_Y=\mu^*(K_X)+E_{\mu}$ where $E_{\mu}$ is the exceptional divisor. Denote by $F_{\mu}$ the strict transform of $F_0$. Then 
$$r_X(K_X^2\cdot F_0)=(r_X\mu^*(K_X)\cdot (K_Y-E_{\mu})\cdot F_{\mu})
=(r_X\mu^*(K_X)\cdot K_Y\cdot F_{\mu})$$
is an integer. 
\end{proof}

\subsection{Technical set up}\label{setup} Keep the same notation as in \ref{m0}. Pick a generic irreducible element $F$ of $|M_{m_0}|$. Assume that $|G|$ is a base point free linear system on $F$. Pick a generic irreducible element $C$ of $|G|$. Since $\pi^*(K_X)|_F$ is nef and big, we may always assume that there exists a positive rational number $\beta>0$ so that $\pi^*(K_X)|_F-\beta C$ is $\bQ$-linearly equivalent to an effective $\bQ$-divisor.  Let $m$ be a positive integer. Set
$$\xi=(\pi^*(K_X)|_F\cdot C),$$
$$\alpha_m=(m-1-\frac{m_0}{a_{m_0}}-\frac{1}{\beta})\xi.$$

\subsection{Some frequently used results}

\begin{thm}\label{key} (see Chen-Chen \cite[2.2, Theorem 2.7]{EXPIII}) Keep the same assumption as in \ref{m0} and \ref{setup}.  Let $m$ be a positive integer.  Let $X$ be a minimal 3-fold of general type with $P_{m_0}\geq 2$ for some integer $m_0>0$.  
The following statements hold:
\begin{itemize}
\item[(i)] One has  the inequality
\begin{equation}\xi\geq \frac{\deg(K_C)}{1+{m_0}/{a_{m_0}}+{1}/{\beta} }.\end{equation}

\item[(ii)] When $\alpha_m>1$, one has the inequality
\begin{equation}m\xi\geq  \deg(K_C)+\roundup{\alpha_m}.\end{equation}
When $\alpha_m>0$ and $C$ is an even divisor, one has the inequality
\begin{equation}m\xi\geq  \deg(K_C)+2\roundup{\frac{\alpha_m}{2}}.\end{equation}
\item[(iii)] Assume that $|mK_{X'}|$ distinguishes different generic irreducible elements of $|M_{m_0}|$ and that, on the generic irreducible  element $F$ of $|M_{m_0}|$, $|mK_{X'}||_F$ distinguishes different generic irreducible elements of $|G|$.  When $\alpha_m>2$, $\varphi_{m,X}$ is birational onto its image. 
\end{itemize}
\end{thm}

\begin{lem}\label{L1} (see Chen-Chen \cite[Lemma 2.4, Lemma 2.5]{EXPIII}) Let $\sigma:S\lrw S_0$ be a birational contraction  from a nonsingular projective surface of general type onto its minimal model $S_0$.  Assume that $(K_{S_0}^2, p_g(S_0))\neq (1,2)$. Then one has 
\begin{itemize}
\item[(1)]  $(\sigma^*(K_{S_0})\cdot C)\geq 2$ for any moving irreducible curve $C$ on $S$ (i.e. $C$ moves in a linear system of positive dimension);

\item[(2)] $(\sigma^*(K_{S_0})\cdot \tilde{C})\geq 2$  for any very general irreducible curve $\tilde{C}$ on $S$.  
\end{itemize}
 \end{lem}

\begin{lem}\label{brat} (\cite[Lemma 2.5]{Chen2014}) Let $S$ be a nonsingular projective surface. Let ${\mathcal M}$ be a nef and big $\bQ$-divisor on $S$ satisfying the following conditions:
\begin{itemize}
\item[(1)] ${\mathcal M}^2>8$;
\item[(2)] $({\mathcal M}\cdot C_{P})\geq 4$ for all irreducible curves $C_{P}$ passing through each very general point $P\in S$.
\end{itemize}
Then the linear system $|K_S+\roundup{{\mathcal M}}|$ separates two distinct points in very general positions.  Consequently $\Phi_{|K_S+\roundup{{\mathcal M}}|}$ is a birational map onto its image.
\end{lem}

\begin{lem}\label{tau} (cf. \cite[Lemma 2.6]{Chen2014}) Let $X'\lrw X$ be a birational morphism from a nonsingular projective model $X'$ onto $X$, which is a minimal projective 3-fold of general type. Assume that $f:X'\lrw \bP^1$ is a fibration with the general fiber $F$. Denote by $\sigma:F\lrw F_0$ the birational contraction onto the minimal model $F_0$.  
The following statements are true:
\begin{itemize}
\item[(1)] Set $\mu_0=\frac{K_X^3}{3(\pi^*(K_X)|_F)^2}$. Then
$$\pi^*(K_X)|_F\geq (\frac{\mu_0}{\mu_0+1}-\varepsilon)\sigma^*(K_{F_0})$$
where $0<\varepsilon\ll 1$ is any rational number.
\item[(2)] If $\pi^*(K_X)\geq \nu_0 F$ for some positive rational number $\nu_0$, then 
$$\pi^*(K_X)|_F\geq \frac{\nu_0}{\nu_0+1}\sigma^*(K_{F_0}).$$
\end{itemize}
\end{lem}
\begin{proof} (1) For any sufficiently large and divisible integer $m$ (such that $m$ is divisible by the canonical index $r(X)$), the Riemann-Roch formula on $X'$ implies
$$P_m(X')=h^0(X', m\pi^*(K_X))\approx \frac{1}{6}K_X^3m^3.$$
On the other hand, the Riemann-Roch on $F$ gives
$$h^0(F,m\pi^*(K_X)|_F)\approx\frac{1}{2}(\pi^*(K_X)|_F)^2m^2$$
as $m$ is sufficiently large. Therefore 
$$P_m(X')>(\mu_0-\delta)mh^0(F,m\pi^*(K_X)|_F)$$
for $m\gg 0$ and very small number $\delta>0$. Consider the restriction maps:
$$H^0(X', M_m-tF)\overset{\theta_t}\lrw V_{m,t}\subset H^0(F, m\pi^*(K_X)|_F)$$ 
where $t\geq 0$, $m$ is sufficiently divisible and $V_{m,t}$ is the image space. Since $\dim (V_{m,t})\leq h^0(F,m\pi^*(K_X)|_F)$ for all $t$, we naturally have 
$$m\pi^*(K_X)-(\mu_0-\delta)mF\geq M_m-(\mu_0-\delta)mF>0$$ for all large and divisible integers $m$ (such that $(\mu_0-\delta)m$ is integral).  Thus, up to $\bQ$-linear equivalence, one has 
$$\pi^*(K_X)\geq (\mu_0-\delta)F.$$
The rest of (1) follows directly from the statement (2).

(2)  We may find a very large and divisible integer $N$ so that $N\pi^*(K_X)$ and $N\nu_0F$ are Cartier divisors and that  
$$N\pi^*(K_X)\geq N\nu_0F$$ 
holds as Cartier divisors. Now we may apply Chen-Chen \cite[Lemma 2.1(ii)]{EXPIII} by replacing 
$(m_0, \Lambda, \theta_{\Lambda})$ with $(N, \pi(N\nu_0F), N\nu_0)$. What we get is
$$\pi^*(K_X)|_F\geq \frac{N\nu_0}{N+N\nu_0}\sigma^*(K_{F_0})=
\frac{\nu_0}{\nu_0+1}\sigma^*(K_{F_0}).$$
\end{proof}

\section{\bf Proof of the main theorem}

\subsection{General setting} \label{gs}
In this section we always assume $X$ to be a minimal 3-fold of general type with $\delta(X)=18$.  Set $m_0=18$ and keep the same notation as in \ref{m0}. We have an induced fibration $f=f_{18}: X'\lrw \Gamma$.  Pick up a general fiber $F$ of $f$. For this kind of 3-folds, the following properties are known:

\begin{itemize}
\item[(R1)]  (see \cite[Theorem 5.1]{EXPIII}) $K_X^3=\frac{1}{1170}$, $P_2(X)=q(X)=0$, $\chi(\OO_X)=2$ and Reid's singularity basket
$$B_X=\{4\times (1,2), (4,9), (2,5), (5,13), 3\times (1,3), 2\times (1,4)\}.$$

\item[(R2)] Direct calculation shows that $P_8=1$, $P_{18}=2$, $P_{19}=0$, $P_{24}=3$, $P_{36}=8$ and $P_m>0$ for all $m\geq 20$. 

\item[(R3)]  (see \cite[Theorem 6.1]{EXPIII}) $p_g(F)=1$ and $r_s(X)\leq 61$. 
\end{itemize}

Set $m_1=24$ and $m_2=36$. For any positive integer $l$, we have defined  $|M_l|=\text{Mov}|lK_{X'}|$. Modulo further birational modification to $\pi$, we may assume that $|M_{18}|$, $|M_{24}|$ and $|M_{36}|$ are all base point free.

Clearly we have $\Gamma\cong \bP^1$ since $q(X)=0$. We will analyze the interactions among $\varphi_{18,X}$, $\varphi_{24,X}$ and $\varphi_{36,X}$. On a general fiber $F$ of $f:X'\lrw \bP^1$, we set $\caL=\pi^*(K_X)|_F$, which is a nef and big $\bQ$-divisor. 

\subsection{A sufficient condition for birationality of $\varphi_{m,X}$}

\begin{lem}\label{l1} Keep the above setting. Then $|mK_{X'}|$ can distinguish different generic irreducible elements of $|M_{18}|$ for all $m\geq 38$. 
\end{lem}
\begin{proof}  By (R2), we have $mK_{X'}\geq M_{18}\sim F$ whenever $m\geq 38$. Since $q(X)=0$, $|F|$ is a rational pencil. Thus $|mK_{X'}|$ naturally distinguishes different fibers of $f$. 
\end{proof}

\begin{lem}\label{l2}   Assume that $\caL\geq \tilde{\beta}\sigma^*(K_{F_0})$ for certain positive rational number $\tilde{\beta}$. Then $|K_F+\roundup{n\caL}|$ gives a birational map whenever
$$n\geq \text{max}\{\rounddown{\sqrt{\frac{8}{\caL^2}}}+1, \frac{2}{\tilde{\beta}}\}.$$
\end{lem}
\begin{proof} First, we have $(n\caL)^2>8$ for such numbers $n$. Secondly, for any very general curve $\tilde{C}$ on $F$, one has
$$(n\caL\cdot \tilde{C})\geq n\tilde{\beta}(\sigma^*(K_{F_0})\cdot \tilde{C})\geq 
2n\tilde{\beta}\geq 4$$
by Lemma \ref{L1}(2) and the fact that $F_0$ is not a $(1,2)$ surface.  Thus $|K_F+\roundup{n\caL}|$ gives a birational map by Lemma \ref{brat}. 
\end{proof}

\begin{prop}\label{k1} Let $X$ be a minimal 3-fold of general type with $\delta(X)=18$. Keep the above notation. Assume that $\caL\geq \tilde{\beta}\sigma^*(K_{F_0})$ for certain positive rational number $\tilde{\beta}$. Then $\varphi_{m,X}$ is birational for all
$$m\geq \text{max}\{\rounddown{\sqrt{\frac{8}{\caL^2}}}+20, \frac{2}{\tilde{\beta}}+19, 38\}.$$
\end{prop}
\begin{proof} For any $n>0$, we have
$$|K_{X'}+\roundup{n\pi^*(K_X)}+F|\lsleq |(n+19)K_{X'}|.$$
By Kawamata-Viehweg vanishing theorem (\cite{KV,VV}), we have  
\begin{eqnarray*}
|K_{X'}+\roundup{n\pi^*(K_X)}+F||_F&=&|K_F+\roundup{n\pi^*(K_X)}|_F|\\
&\lsgeq& |K_F+\roundup{n\caL}|.
\end{eqnarray*}
Now the statement clearly follows from Lemma \ref{l1} and Lemma \ref{l2}. 
\end{proof}

\begin{rem}  In practice we may take $\tilde{\beta}$ to be $\frac{1}{19}$ by Lemma \ref{tau} while applying Proposition \ref{k1}.
\end{rem}

\subsection{The case when $|18K|$ and $|24K|$ are composed of the same pencil}

\begin{thm}\label{same} Let $X$ be a minimal 3-fold of general type with $\delta(X)=18$. Assume that $|18K_X|$ and $|24K_X|$ are composed of the same pencil of surfaces. Then $\varphi_{m,X}$ is birational for all $m\geq 53$. 
\end{thm}
\begin{proof} By our assumption, both $\varphi_{18}$ and $\varphi_{24}$ induces the same fibration $f:X'\lrw \bP^1$. 
Since $P_{24}(X)=3$, we may write
$$24\pi^*(K_X)\sim 2F+{E}_{24}'$$
where ${E}_{24}'$ is an effective $\bQ$-divisor. 
By Lemma \ref{tau}(2), we have 
$$\pi^*(K_X)|_F\geq \frac{1}{13}\sigma^*(K_{F_0})$$
which means that we have $\tilde{\beta}=\frac{1}{13}$. 
Clearly we have $\caL^2\geq \frac{1}{13^2}$. 

In fact, one may have a better estimation for $\caL^2$. 
We have $m_1=24$, $a_{m_1}=2$. On the general fiber $F$, set $|\tilde{G}|=|2\sigma^*(K_{F_0})|$, which is base point free by the surface theory. Pick a generic irreducible element $C'$ of $|\tilde{G}|$. Clearly $C'$ is an even divisor on $F$. Replacing $(m_0, a_{m_0},f,F,|G|, C,\beta)$ with $(m_1,a_{m_1},f,F,|\tilde{G}|,C',2\tilde{\beta})$ while applying Theorem \ref{key}, one easily obtains that $(\caL\cdot C')\geq \frac{1}{5}$. Thus 
$$\caL^2=(\pi^*(K_X)|_F)^2\geq \frac{1}{26}(\caL\cdot C')\geq \frac{1}{130}.$$  

By Proposition \ref{k1}, $\varphi_{m,X}$ is birational for all $m\geq 53$. 
\end{proof}

In this case, we can find the constraint for $p_g(F)$. 

\begin{cor}  Under the same assumption as that of Theorem \ref{same},  one has $(K_{F_0}^2, p_g(F_0))=(1,1)$. 
\end{cor}
\begin{proof} Suppose $K_{F_0}^2\geq 2$. Then we have 
$\caL^2\geq \frac{2}{13^2}$. This implies 
$$K_X^3\geq \frac{1}{12}\caL^2\geq \frac{1}{1014}>\frac{1}{1170},$$
a contradiction. 
\end{proof}

\subsection{The case when $|18K|$ and $|24K|$ are not composed of the same pencil}

We set $|G|=|M_{24}|_F|$. By our assumption, $|G|$ is base point free. Pick a generic irreducible element $C$ in $|G|$.  Recall that $\xi=(\caL\cdot C)$.  Clearly we may take $\beta=\frac{1}{24}$. 

\begin{lem}\label{es}   One has $\caL^2\geq \frac{1}{195}$. 
\end{lem}
\begin{proof}  Assume that 
$(\pi^*(K_X)|_F)^2\leq \frac{11}{2340}$.  Then we have
$$\mu_0=\frac{K_X^3}{3(\pi^*(K_X)|_F)^2}\geq \frac{2340}{33\cdot 1170}=\frac{1}{16.5},$$
which implies, by Lemma \ref{tau} and Lemma \ref{L1}, that
$$\xi=(\pi^*(K_X)|_F\cdot C)\geq \frac{1}{17.5}(\sigma^*(K_{F_0})\cdot C)\geq \frac{2}{17.5}$$
and, on the other hand, 
$$(\pi^*(K_X)|_F)^2\geq \frac{1}{24}\xi\geq \frac{1}{210}>\frac{11}{2340},$$
a contradiction.  Thus we have $\caL^2>\frac{11}{2340}$.

Noting that $r_X\cdot (\pi^*(K_X)|_F)^2$ is integral by Lemma \ref{integ} and $r_X=2340$, we see 
$$\caL^2=(\pi^*(K_X)|_F)^2\geq \frac{12}{2340}=\frac{1}{195}.$$
\end{proof}

We have actually proved the following weaker result:
\begin{cor} Let $X$ be a minimal 3-fold of general type with $\delta(X)=18$.  Then $\varphi_{m,X}$ is birational for all $m\geq 59$. 
\end{cor}
\begin{proof} By Theorem \ref{same}, we may assume that $|18K|$ and $|24K|$ are not composed of the same pencil.  The statement directly follows from Proposition \ref{k1} and Lemma \ref{es}. 
\end{proof}

In order to prove the birationality of $\varphi_{57,X}$, we start to study the behavior of $|36K_{X'}|$. Recall that we have 
$|M_{36}|=\text{Mov}|36K_{X'}|$. Consider the natural map
$$H^0(X', M_{36})\overset{\theta_{36}}\lrw V_{36}\subseteq H^0(F, M_{36}|_F)$$
where $V_{36}=\text{Im}(\theta_{36})$. 

\begin{thm}\label{v5} Let $X$ be a minimal 3-fold of general type with $\delta(X)=18$. Assume that $|18K_{X}|$ and $|24K_{X}|$ are not composed of the same pencil. If $\dim_k(V_{36})\geq 5$, then $\varphi_m$ is birational for all $m\geq 57$.
\end{thm}
\begin{proof}   Clearly we have $h^0(F, M_{36}|_F)\geq 5$. Set $|G_{36}|=|M_{36}|_F|$. 
By our assumption, $|G_{36}|$ is base point free.  
Recall that we have already another curve family $|C|$ on $F$. 

We consider the natural restriction map
$$H^0(F, G_{36})\lrw W_{36}\subset H^0(C, D_{36})$$
where $W_{36}$ is the image linear space and $D_{36}=G_{36}|_{C}$.
\medskip

{\bf Case 1}.  If $\dim(W_{36})\geq 4$, we have $h^0(C, D_{36})\geq 4$.
By Riemann-Roch formula and the Clifford theorem, one easily knows that $\deg(D_{36})\geq 5$ since $g(C)\geq 2$.  So
$$\caL^2=(\pi^*(K_X)|_F)^2\geq \frac{1}{36\cdot 24} (M_{36}|_F\cdot M_{24}|_F)\geq \frac{5}{36\cdot 24}=\frac{1}{172.8}.$$
\medskip

{\bf Case 2}.  If $\dim(W_{36})\leq 3$, we have $h^0(F, G_{36}-C)\geq 2$.  Pick up a generic irreducible element $C''$ in $\text{Mov}|G_{36}-C|$. One has $G_{36}\geq C+C''$. 
Then, by Lemma \ref{L1} and Lemma \ref{tau}, one has 
$$\big(\pi^*(K_X)|_F\cdot G_{36}\big)\geq \frac{1}{19}\Big(\sigma^*(K_{F_0})\cdot (C+C'')\Big)\geq \frac{4}{19}.$$
Thus one has
$$\caL^2=(\pi^*(K_X)|_F)^2\geq \frac{1}{36} \big(\pi^*(K_X)|_F\cdot G_{36}\big)\geq \frac{1}{171}.$$

By Proposition \ref{k1}, $\varphi_{m,X}$ is birational for all $m\geq 57$. 
\end{proof}

\begin{thm}\label{v4} Let $X$ be a minimal 3-fold of general type with $\delta(X)=18$. Assume that $|18K_{X'}|$ and $|24K_{X'}|$ are not composed of the same pencil. If $\dim(V_{36})\leq 4$, then $\varphi_m$ is birational for all $m\geq 57$.
\end{thm}

\begin{proof}  Since $h^0(X', M_{36})=P_{36}(X)=8$, we have $h^0(X', M_{36}-F)\geq 4$. 
We may modify our previous $\pi$  so that $\text{Mov}|M_{36}-F|$ is also base point free.  Set $|M_{36,-1}|=\text{Mov}|M_{36}-F|$. 
\medskip

{\bf Case I}.  Assume that $|M_{36,-1}|$ is composed of a pencil of surfaces. Noting that $q(X)=0$, such a pencil $|M_{36,-1}|$ must be over the rational curve (namely, a rational pencil). So 
$$36\pi^*(K_X)-F\geq M_{36,-1}\geq 3F_1$$
where $F_1$ is a generic irreducible element of $|M_{36,-1}|$.

{\bf Subcase (I-1)}. Suppose that both $|F|$ and $|F_1|$ are the same pencil, we have $36\pi^*(K_X)\geq 4F$ and, by Lemma \ref{tau},  the inequality
$$\pi^*(K_X)|_F\geq  \frac{1}{10}\sigma^*(K_{F_0}). $$ 
Thus 
$$\caL^2=(\pi^*(K_X)|_F)^2\geq  \frac{1}{100}.$$
Proposition \ref{k1}  implies that $\varphi_{m}$ is birational for all $m\geq 48$.

{\bf Subcase (I-2)}.  If  $|F|$ and $|F_1|$ are not the same pencil, then $|F_1|_F|$ is moving and we pick a generic irreducible element $C_1$ in $|F_1|_F|$.  We are going to study the birationality of $\varphi_{m,X}$ directly. We consider the following linear system on $X'$: 
$$|K_{X'}+\roundup{s\pi^*(K_X)}+F+F_1|\lsleq|(s+37)K_{X'}|$$
where $s\in \bZ_{>0}$.  By Kawamata-Viehweg vanishing theorem, we have 
\begin{eqnarray}
&&|K_{X'}+\roundup{s\pi^*(K_X)}+F+F_1||_F\notag\\
&=& |K_F+\roundup{s\pi^*(K_X)}|_F+F_1|_F|\notag\\
&\lsgeq&|K_F+\roundup{s\pi^*(K_X)|_F}+C_1|.\label{1}
\end{eqnarray}
Whenever $s\geq 19$, $(K_{X'}+\roundup{s\pi^*(K_X)})|_F$ is effective.  So 
$$|K_{X'}+\roundup{s\pi^*(K_X)}+F+F_1||_F$$
can distinguish different generic irreducible elements of $|F_1|_F|$.  Applying the vanishing theorem once more, we have
\begin{eqnarray*}
&&|K_F+\roundup{s\pi^*(K_X)|_F}+C_1||_{C_1}\\
&=&|K_{C_1}+\roundup{s\pi^*(K_X)|_F}|_{C_1}|\\
&=&|K_{C_1}+D_1|.\end{eqnarray*}
Noting that 
$$(\pi^*(K_X)|_F\cdot C_1)\geq \frac{2}{19},$$
we have 
$$\deg(D_1)\geq s\big(\pi^*(K_X)|_F\cdot C_1\big)>2$$
whenever $s\geq 20$. 
By Lemma \ref{l1}, we have seen that $\varphi_m$ is birational for all $m\geq 57$. 
\medskip

{\bf Case II}.  Assume that $|M_{36,-1}|$ is not composed of a pencil of surfaces. Pick a generic irreducible element $S_{-1}\in |M_{36,-1}|$. We have
$$36\pi^*(K_X)\geq F+S_{-1}.$$
Set $|G_{-1}|=|S_{-1}|_F|$.  We use the parallel argument to that for Subcase (I-2). 
In fact, we have 
$$|K_{X'}+\roundup{s\pi^*(K_X)}+F+S_{-1}|\lsleq |(s+37)K_{X'}|$$
for any  $s\in \bZ_{>0}$.  By Kawamata-Viehweg vanishing theorem, we have 
\begin{eqnarray}
&&|K_{X'}+\roundup{s\pi^*(K_X)}+F+S_{-1}||_F\notag\\
&=& |K_F+\roundup{s\pi^*(K_X)}|_F+S_{-1}|_F|\notag\\
&\lsgeq&|K_F+\roundup{s\pi^*(K_X)|_F}+G_{-1}|.\label{2}
\end{eqnarray}
No matter whether $|G_{-1}|$ is composed of a pencil or not,  $|K_{X'}+\roundup{s\pi^*(K_X)}+F+S_{-1}||_F$ can distinguish different generic irreducible elements of $|G_{-1}|$ whenever $s\geq 20$ (which is an easy exercise as an application of Kawamata-Viehweg vanishing theorem!). 

{}Finally, we pick a generic irreducible element $C_{-1}$ of $|G_{-1}|$.  Applying the vanishing theorem once more, we have
\begin{eqnarray*}
&&|K_F+\roundup{s\pi^*(K_X)|_F}+G_{-1}||_{C_{-1}}\\
&=&|K_{C_{-1}}+\roundup{s\pi^*(K_X)|_F}|_{C_{-1}}+(G_{-1}-C_{-1})|_{C_{-1}}|\\
&\lsgeq&|K_{C_{-1}}+D_{-1}|\end{eqnarray*}
where 
$$D_{-1}=\roundup{s\pi^*(K_X)|_F}|_{C_{-1}}$$
since  $(G_{-1}-C_{-1})|_{C_{-1}})\geq 0$. Noting that, by Lemma \ref{L1},  
$$\big(\pi^*(K_X)|_F\cdot C_{-1}\big)\geq \frac{2}{19},$$
we have 
$$\deg(D_{-1})\geq s\Big(\pi^*(K_X)|_F\cdot C_{-1}\Big)>2$$
whenever $s\geq 20$. 
By Lemma \ref{l1}, we have seen that $\varphi_m$ is birational for all $m\geq 57$. 
\end{proof}

\subsection{Proof of Theorem \ref{57}}
\begin{proof} Theorem \ref{57} follows directly from Theorem \ref{same}, Theorem \ref{v5} and Theorem \ref{v4}. \end{proof}

\subsection{Proof of Corollary \ref{cc}}
\begin{proof} Recall the following: 
\begin{thm}\label{15} (see Chen--Chen \cite[Theorem 6.2]{EXPIII})  Let $X$ be a minimal 3-fold of general type with $\delta(X)\leq 15$. Then $\varphi_m$ is birational for all $m\geq 56$.
\end{thm}

Now Corollary \ref{cc}  follows directly from Theorem \ref{15},  \cite[Theorem 1.4(2)]{EXPIII} ($\delta(X)\neq 16$,$17$) and Theorem \ref{57}. \end{proof}

\subsection{Open question} It is natural and interesting to ask: is $r_3\leq 57$ optimal?

\vskip 0.8cm
\noindent{\bf Acknowledgment}.  The author was supported by National Natural Science Foundation of China (\#11571076,\#11231003, \#11421061) and Program of Shanghai Academic Researcher Leader (Grant no. 16XD1400400).  During the preparation of this paper, the author benefited a lot from discussions with Jungkai Chen and Yongnam Lee. The author would like to thank NCTS (Taipei Office) and KAIST for hosting his visits in 2014 and 2015.  Finally the author appreciates the helpful suggestion of a referee.

\end{document}